\documentclass{amsart}
\usepackage{mathptmx, bbm, amscd, amssymb, enumerate, colonequals, mathdots, xcolor,comment,graphicx,lineno,xspace,mathtools}

\usepackage{color}
\usepackage{tikz-cd}
\definecolor{chianti}{rgb}{0.6,0,0}
\definecolor{meretale}{rgb}{0,0,.6}
\definecolor{leaf}{rgb}{0,.35,0}
\usepackage[colorlinks=true, pagebackref, hyperindex, citecolor=meretale, urlcolor=leaf, linkcolor=chianti]{hyperref}
\usepackage[all]{xy}
\usepackage{paralist}
\usepackage{mathrsfs}
\usepackage{enumitem}

\newtheorem{theorem}{Theorem}[section]
\newtheorem{lemma}[theorem]{Lemma}
\newtheorem{corollary}[theorem]{Corollary}
\newtheorem{proposition}[theorem]{Proposition}
\newtheorem{maintheorem}{Theorem}

\theoremstyle{definition}
\newtheorem{definition}[theorem]{Definition}

\newtheorem{example}[theorem]{Example}
\newtheorem{remark}[theorem]{Remark}
\newtheorem{question}[theorem]{Question}
\numberwithin{equation}{theorem}
\DeclarePairedDelimiter\floor{\lfloor}{\rfloor}

\def\ara{\operatorname{ara}}

\def\height{\operatorname{ht}}

\def\image{\operatorname{image}}

\def\ctens{\operatorname{\widehat{\otimes}}}
\def\ctor{\operatorname{\widehat{Tor}}}
\def\dtens{\operatorname{\overset{\mathbb{L}}{\otimes}}}
\def\bight{\operatorname{bight}}
\def\cd{\operatorname{cd}}

\def\Spec{\operatorname{Spec}}
\def\Supp{\operatorname{Supp}}

\def\CC{\mathbb{C}}

\def\phi{\varphi}

\def\to{\longrightarrow}

\def\onto{{\,\,\longrightarrow\hspace{-1.8ex}\rightarrow}\,\,}

\begin{document}
\title[Vanishing of local cohomology in unramified mixed characteristic]{Vanishing of local cohomology in unramified\\mixed characteristic}

\author[Batavia]{Manav Batavia}
\address{Department of Mathematics, Purdue University, 150 N University St., West Lafayette, IN~47907, USA}
\email{mbatavia@purdue.edu}

\dedicatory{Dedicated to Gennady Lyubeznik for his immense contributions to local cohomology.}

\begin{abstract}
     Given an ideal $I$ in a regular local ring $A$, the cohomological dimension of $I$ in $A$ is the index of the highest non-vanishing local cohomology of $A$ supported at $I$. Determining effective upper bounds on the cohomological dimension in terms of topological invariants of $\Spec(A/I)$ is a central problem in commutative algebra. In equal characteristic, Faltings proved in 1980 a general bound on the cohomological dimension of an ideal in terms of its big height. In this article, we extend Faltings’ result to the unramified mixed characteristic setting and show that the resulting bound is sharp.
\end{abstract}
\maketitle

\section{Introduction}

Local cohomology, introduced by Grothendieck in the early 1960s, plays a pivotal role in commutative algebra and algebraic geometry. Its vanishing behavior reflects subtle geometric and topological properties of the underlying scheme and has been studied extensively in the work of Grothendieck, Hartshorne, Faltings, and others. Understanding when local cohomology modules vanish remains a fundamental and delicate problem. 

Let $A$ be a commutative Noetherian ring of Krull dimension $d$ and $I$ be an ideal of $A$. The \textit{cohomological dimension} of $I$ in $A$ is $$\cd(A,I):=\max\{n\:|\:H^n_I(A)\neq 0\}.$$ Alternatively, $\cd(A,I)$ is the least integer $n$ such that $H^q_I(M)=0$ for all $A$-modules $M$ and $q>n$ \cite[Theorem 9.6]{24h}. As $H^q_I(M)=0$ for all $A$-modules $M$ and $q>d$ \cite{Gr}, $\cd(A,I)$ always exists and is bounded above by $d$.

The computation of better upper bounds on cohomological dimension has been of great interest over the past few decades. It is well-known that $$\cd(A,I)\leq\ara(I),$$ where $\ara(I)$, the $\textit{arithmetic rank}$ of $I$, is the least number of equations required to generate $I$ up to radical. However, the arithmetic rank of an ideal is difficult to compute and can be far from the cohomological dimension \cite{BS,Bar,JPSW,BMMP}. 

A complementary approach is to bound the cohomological dimension using topological invariants of the spectrum of $A/I$. The most famous result in this direction is the Hartshorne--Lichtenbaum Vanishing Theorem \cite{Ha1}, which states that for a complete local domain $(A,\mathfrak{m})$, we have $\cd(A,I)\leq d-1$ if and only if $I$ is not $\mathfrak{m}$-primary. Furthermore, the Second Vanishing Theorem \cite{Ogus,PS,HunLyu,Zhang} asserts that in equal characteristic or unramified mixed characteristic, for a complete regular local ring $(A,\mathfrak{m})$ with separably closed residue field, we have $\cd(A,I)\leq d-2$ if and only if $\dim(A/I)\geq 2$ and the punctured spectrum of $A/I$ is connected. Peskine and Szpiro \cite{PS} proved that if $(A,\mathfrak{m)}$ is a regular local ring of equal characteristic $p$ and $A/I$ is Cohen-Macaulay, then $\cd(A,I)=\height(I).$ Beyond the results discussed here, further vanishing statements have been proved in various settings \cite{MR3078644,DaoTakagi, Burke}, though the conditions required become progressively harder to check. Ogus \cite{Ogus}, Hartshorne and Speiser \cite{HarSpe}, Lyubeznik \cite{Lyu3}, Mustață and Popa \cite{MuPo}, and Reichelt, Saito and Walther \cite{RTW} obtained concrete formulae for cohomological dimension in equal characteristic in terms of singularity invariants, see also \cite{BBLSZ}. However, these methods require the understanding of de Rham cohomologies, Hodge filtrations, or Frobenius structures, and in practice they can be difficult to compute.

In equal characteristic, Faltings \cite{Faltings} established a general bound on cohomological dimension in terms of the \textit{big height} of an ideal. The big height of an ideal is the maximum of the heights of all its minimal primes. In this article, we extend Faltings' result to the unramified mixed characteristic setting.

\begin{maintheorem}\label{thm:maintheorem}(cf. Theorem \ref{thm:cdbound})
Let $A$ be a $d$-dimensional unramified regular local ring of mixed characteristic $(0,p)$. Let $I$ be a nonzero proper ideal of $A$ with big height $c$. Then $$\cd(A,I)\leq d-\floor*{\frac{d-1}{c}}.$$  
\end{maintheorem}

We also establish an inductive criterion for the vanishing of local cohomology modules. See Theorem \ref{thm:induction} for a more general statement than the one stated below.

\begin{maintheorem}\label{thm:indcrit}(cf. Theorem \ref{thm:induction})
    Let $(A,\mathfrak{m})$ be an unramified regular local ring of mixed characteristic $(0,p)$ and $I$ be an ideal of $A$ of big height $c$. Suppose $p$ is a nonzerodivisor on $A/I.$
    
    Let $n>c$ be an integer. Assume that for all integers $s$, with $1\leq s\leq c-1,$ and for all $q\geq n-s$, $H^q_{IA_P}(A_P)=0$ for all $P\in \Spec(A)$ such that $I\subseteq P$ and $\dim(A/P)\geq s+1.$ Then $H^q_I(A)=0$ for all $q\geq n.$
\end{maintheorem}

We now outline the strategy for proving our results.

\begin{enumerate}
    \item In equal characteristic, Faltings proved a criterion \cite[Satz 1]{Faltings} of the nature of Theorem \ref{thm:indcrit}, and used it as an inductive tool to establish an upper bound on $\cd(A,I).$ We generalize Faltings' criteria to the unramified mixed characteristic setting under the assumption that $p$ is a nonzerodivisor on $A/I$.\footnote{This assumption can be further relaxed: see Theorem \ref{thm:induction} for details.}
    
    In the proof, we reduce the problem to the case of a local cohomology module supported only at the maximal ideal. We handle this remaining case by introducing an auxiliary tensor construction over the base discrete valuation ring and analyzing local cohomology with respect to the diagonal ideal. This allows us to exploit precise dimension bounds and numerical inequalities to force the collapse of a natural spectral sequence, yielding a vanishing statement in the auxiliary setting.
    We then leverage a result of Zhou on the injective dimension of local cohomology modules in unramified mixed characteristic \cite[Theorem 5.1]{Zhou}, together with the finiteness of Bass numbers for such modules established in \cite{Lyu4, NB}.
    \item We split our proof of Theorem \ref{thm:maintheorem} into three cases. If $p$ is a nonzerodivisor on $A/I$, we can use the inductive criterion above to prove that the local cohomology modules $H_I^n(A)$ vanish for $n$ greater than the desired bound. If $p\in I,$ we reduce our analysis to $A/p$, a regular local ring of equal characteristic $p$, and employ the already known bound in equal characteristic.
    \item If $p$ is a zerodivisor on $I$, we build on the analysis of the preceding two cases to once again reduce to equal characteristic $p$. The guiding observation is that $H^n_I(A)$ is $p$-power torsion for $n$ exceeding the desired bound (Lemma \ref{lemma:p_zd_on_H_I^n}). 
    
    At this stage, the argument must be refined to control the behavior of the big height of the ideal $I(A/p).$ A systemic analysis reveals that all cases can be resolved except when $\bight(I)$ divides $\dim(A)-1$. In this remaining case, a key input is the structural description of the critical local cohomology module $H^q_J(A/p)$ \cite[Theorem 4.1]{Lyu1}, which allows us to reduce to the previously treated situation in which $p$ is a nonzerodivisor on $A/I$.

\end{enumerate}

The paper is organized as follows. In Section \ref{section: InductionTheorem}, we collect the principal properties and a number of subtle features of the big height of an ideal. We then prove the inductive criteria for the vanishing of local cohomology in the unramified mixed characteristic setting (Theorem \ref{thm:induction}). In Section \ref{section:cdbound}, we use Theorem \ref{thm:induction} to prove the upper bound on cohomological dimension (Theorem \ref{thm:cdbound}). We also address the sharpness of our bound (Example \ref{example:sharpness}) and illustrate how the bound can be used to compute the exact cohomological dimension of suitably chosen ideals. In Section \ref{section:questions}, we collect natural questions suggested by our results, and comment on current progress and related directions.

\section{The Induction Theorem}\label{section: InductionTheorem}

In this section, we prove a technical theorem (Theorem \ref{thm:induction}) which will serve as an essential inductive tool while proving our bounds on cohomological dimension. This theorem is an extension of \cite[Satz 1]{Faltings} to the unramified mixed characteristic setting. We begin by introducing the \textit{big height} of an ideal.


\begin{definition}
    Let $A$ be a Noetherian ring. Given an ideal $I$ of $A$, the \textit{big height} of $I$ is $$\bight(I):=\max\{\height(P)|\: P\text{ is a minimal prime of }I\}.$$ 
\end{definition}

\begin{remark} \label{rmk:bight}
    With notations as above,
    \begin{enumerate}
        \item If $I=I_1\cap I_2,$ then $\bight(I)\leq\max\{\bight(I_1),\bight(I_2)\}$. Equality may not hold: for instance, if $A=\CC[x,y,z]$, $I_1=(x)$ and $I_2=(x,y)\cap (z)$, then $\bight(I_1)=1$ and $\bight(I_2)=2$. However, $I=(xz)$ and $\bight(I)=1.$
        \item It is possible that $I\subseteq J$ and $\bight(I)>\bight(J).$ For instance, in $A=\CC[x,y,z]$, if $I=(x)\cap(y,z)$ and $J=(x)$, then $\bight(I)=2$ and $\bight(J)=1$.
        \item If $P$ is a prime ideal of $A$ containing $I$, then $\bight(IA_P)\leq\bight(I).$
        \item Observe that $\cd(A,I)\geq\bight(I).$ Indeed, let $c=\bight(I)$ and $P$ be a minimal prime of $I$ of height $c$. Then $IA_P$ is primary to the maximal ideal $PA_P$ of $A_P$ and $H^c_{IA_P}(A_P)\neq0$ by a result of Grothendieck \cite[Proposition 6.4]{Gr}; alternatively, see \cite[Theorem 9.3]{24h}. Hence, $H_I^c(A)\neq0$ and $\cd(A,I)\geq c.$
    \end{enumerate}
\end{remark}

\begin{lemma}\label{lemma:fflat}
    Let $A$ be a Noetherian ring, $I$ be an ideal of $A$ and $B$ be a faithfully flat extension of $A$. Then $\bight(I)=\bight(IB).$ 
\end{lemma}
\begin{proof}
    We claim that the set of minimal primes of $IB$ is exactly the union of the sets of minimal primes of $PB$, where $P$ varies over the minimal primes of $I.$
    Let $Q$ be a minimal prime of $IB$, and let $P=Q\cap A$. If $P$ is not a minimal prime of $A$, then $P\supsetneq P'\supseteq I$, where $P'\in\Spec(A).$ By the going down theorem, there exists $Q'\in\Spec(B)$ such that $Q\supsetneq Q'$ and $Q'\cap A=P'$. Then $Q\supsetneq Q'\supseteq IB$, a contradiction. 

    Now, let $P$ be a minimal prime of $I$. We note that every minimal prime $Q$ of $PB$ is also a minimal prime of $IB$. Indeed, if not, $Q\supsetneq Q'\supseteq IB$ for some $Q'\in\Spec(B).$ However, this implies $P=Q\cap A\supseteq Q'\cap A\supseteq I$. As $P$ is a minimal prime of $I$, we have $Q'\cap A=P$. Then $Q\supsetneq Q'\supseteq PB$, a contradiction.

    It remains to be shown that if $Q$ is a minimal prime of $PB$ for $P\in \Spec(A)$, then $\height(P)=\height(Q).$ As $B_Q$ is a faithfully flat $A_P$-algebra, by \cite[Theorem A.11]{BrunsHerzog},
    \begin{align*}
        \dim(B_Q)&=\dim(A_P)+\dim(B_Q/PB_Q)\\
                 &=\dim(A_P).
    \end{align*}
    Therefore, $\height(P)=\height(Q).$
\end{proof}

We now prove three lemmas which will be useful in the proof of Theorem \ref{thm:induction}.

\begin{lemma}\label{lemma:colon}
    Let $I$ be a radical ideal of a Noetherian ring $A$ and $x \in A$. Let $P_1,\dots,P_t$ be the minimal primes of $I$. Suppose \begin{align*}
        I_1=\bigcap_{x\in P_i}P_i \; \text{ and } \;
        I_2=\bigcap_{x\notin P_i}P_i.
    \end{align*} Then $(I:x)=I_2$ and $(I:(I:x))=I_1.$
\end{lemma}
\begin{proof}
    As $x\in I_1$ and $I=I_1\cap I_2$, it is straightforward to see that $I_2\subseteq (I:x).$ Suppose $y\in (I:x)$. Then $xy\in I_2$, which implies $y\in I_2$ as $x$ is a nonzerodivisor on $A/I_2$. Hence, $(I:x)=I_2$.

    It remains to show $(I:I_2)=I_1$. Suppose $z\in (I:I_2).$ Then $zI_2\subseteq I_1$, which implies $zI_2\subseteq P$ for every minimal prime $P$ of $I_1$. However, $I_2\not\subseteq P$ by hypothesis and therefore $z\in P$ for every minimal prime $P$ of $I_1$. Hence, $z\in I_1$. The other containment is trivial.  
\end{proof}

\begin{lemma}\label{lemma:radicalsubset}
    Let $I$ and $J$ be ideals in a Noetherian ring $A$. Then $$\sqrt{(I:(I:J))}\subseteq \sqrt{(\sqrt{I}:(\sqrt{I}:J))}.$$
\end{lemma}
\begin{proof}
    Let $x\in\sqrt{(I:(I:J))}$ and $y\in(\sqrt{I}:J)$. There exist $n,m\in\mathbb{N}$ such that $x^n\in(I:(I:J))$ and $(yJ)^m\in I.$ Since $y^mJ^{m-1}\subseteq (I:J),$ we have $x^ny^mJ^{m-1}\subseteq I$. If $m>1$, we can conclude that $x^ny^mJ^{m-2}\subseteq I:J$ and hence, $x^{2n}y^mJ^{m-2}\subseteq I.$ Proceeding similarly, it follows that $$x^{mn}y^m\in I,$$ which implies $x^ny\in \sqrt{I}.$ Then $x^n\in (\sqrt{I}:(\sqrt{I}:J))$.
\end{proof}

\begin{remark}
    The above containment may be proper. For example, in $A=\CC[x,y]$, if $I=(x^2,xy)$ and $J=(y),$ then $\sqrt{(I:(I:J))}=(x,y)$ and $\sqrt{(\sqrt{I}:(\sqrt{I}:J))}=A.$
\end{remark}

\begin{lemma}\label{lemma:fflat0dimfiber}
Let $(A,\mathfrak{m})$ be a Noetherian local ring, and let $I$ be an ideal of $A$. Suppose that
$B$ is a faithfully flat local extension of $A$ with zero-dimensional fiber, that is, $\dim(B/\mathfrak{m}B)=0.$ If
$H^i_{IA_P}(M_P)=0$ for an $A$-module $M$ and for all $P\in\Spec(A)$ with
$\dim(A/P)\geq r$, then
\[
H^i_{IB_Q}(M\otimes_A B_Q)=0
\]
for all $Q\in\Spec(B)$ with $\dim(B/Q)\geq r$.
\end{lemma}

\begin{proof}
Let $Q$ be a prime of $B$ with $\dim(B/Q)\geq r$, and set $P=Q\cap A$. Since $B/PB$ is faithfully flat over $A/P$ and $B/\mathfrak mB$ is $0$-dimensional, we have
\[
\dim(A/P)=\dim(B/PB)\geq \dim(B/Q)\geq r.
\]
Therefore $H^i_{IA_P}(M_P)=0$, and hence
\[
H^i_{IB_Q}(M\otimes_A B_Q)
= H^i_{IA_P}(M_P)\otimes_{A_P} B_Q = 0.
\qedhere\]
\end{proof}

\begin{definition}
    A regular local ring $(A,\mathfrak{m})$ of mixed characteristic $(0,p)$ is said to be \textit{unramified} if $p\in\mathfrak{m}\setminus\mathfrak{m}^2.$
\end{definition}

\noindent We are now ready to prove our inductive criterion for vanishing of local cohomology.

\begin{theorem}\label{thm:induction}
    Let $(A,\mathfrak{m},k)$ be an unramified regular local ring of mixed characteristic $(0,p)$, and let $I$ be an ideal of $A$. Suppose $\height(I)<\height(I:(I:p))$ (for example, this holds if $p$ is a nonzerodivisor on $A/I)$.\footnote{We use the convention that $\height(A)=\infty.$}
    
    Set $c=\bight(I)$. Let $n>c$ be an integer. Assume that for all integers $s$, with $1\leq s\leq c-1,$ and for all $q\geq n-s$, $H^q_{IA_P}(A_P)=0$ for all $P\in \Spec(A)$ such that $I\subseteq P$ and $\dim(A/P)\geq s+1.$ Then $H^q_I(A)=0$ for all $q\geq n.$
\end{theorem}

\begin{proof}
    We first observe that 
    \begin{align*}
        \height(\sqrt{I})=\height(I)&<\height(I:(I:p))\\
        &= \height\big(\sqrt{(I:(I:p))}\big)\\
        &\leq \height\bigg(\sqrt{(\sqrt{I}:(\sqrt{I}:p))}\bigg)\\
        &= \height\big(\sqrt{I}:(\sqrt{I}:p)\big)
    \end{align*} by  Lemma \ref{lemma:radicalsubset}. 
    As local cohomology and big height of an ideal only depend on its radical, we may reduce to the case where $I$ is radical. 
    By descending induction on $n$, we may assume $H^q_I(A)=0$ for all $q>n$. We need to prove $H^n_I(A)=0.$ By Cohen's structure theorem, Lemma \ref{lemma:fflat} and Lemma \ref{lemma:fflat0dimfiber}, as completion is faithfully flat with zero-dimensional fiber, we may assume that $A=V[[X_2,\dots,X_e]],$ where $V$ is a complete DVR with uniformizer $p$. We induce on $\dim(A/I).$ If $\dim(A/I)=0$, then $c=\dim(A)$ and $H^n_I(A)=0$ as $n>c.$

    Now, assume $d:=\dim(A/I)>0.$ We claim that $H^n_I(A)$, if nonzero, is supported only at the maximal ideal $\mathfrak{m}$ of $A$. Suppose not; let $P$ be minimal in $\Supp H_I^n(A).$ Note that $I\subseteq P.$ We have two cases: 

    \begin{enumerate}[label=(\roman*)]
        \item $p \notin P$.

        \noindent In this case, $A_P=(V[[X_2,\dots,X_e]])_P$ contains a field of characteristic $0$. Let $C=\widehat{A_P}$. We know that $c=\bight(I)\geq \bight(IC)$ by Remark \ref{rmk:bight} and Lemma \ref{lemma:fflat}. We use \cite[Satz 1]{Faltings} to prove that $H^n_{IC}(C)=0,$ which would imply $H^n_{IA_P}(A_P)=(H^n_I(A))_P=0,$ a contradiction to the choice of $P$.
        
        Let $s$ be an integer such that $1\leq s\leq \bight(IC)-1.$ Let $Q\in \Spec(C)$ such that $IC\subseteq Q$ and $\dim(C/Q)\geq s+1.$ Let $Q'=Q\cap A$ be the contraction of $Q$ to $A$. We have $$\dim(A/Q')>\dim(A_P/Q'A_P)\geq\dim(C/Q)\geq s+1.$$ As $\dim(A/Q')>s+1,$ by the hypothesis of our theorem, $H^q_{IA_{Q'}}(A_{Q'})=0$ and thus, $H^q_{IC_Q}(C_Q)=0$ for all $q\geq n-s$, since $A_{Q'}\to C_Q$ is faithfully flat. By \cite[Satz 1]{Faltings}, $H^q_{IC}(C)=0$ for all $q\geq n$.

        \item $p\in P$.

        We define $C=\widehat{A_P}$ as above and conclude again that for all integers $s$ with $1\leq s\leq \bight(IC)-1$, and for all $q\geq n-s$, we have $H^q_{IC_Q}(C_Q)=0$ for all $Q\in \Spec(C)$ such that $IC\subseteq Q$ and $\dim(C/Q)\geq s+1.$ Note that in this case, $C$ is still an unramified regular local ring of mixed characteristic $(0,p)$ because $P^{(2)}\subseteq \mathfrak{m}^2$ \cite[Theorem 2.11]{SymPowers}. Further, $\height(IC)=\height(I)$ and $\height(IC:(IC:p))=\height(I:(I:p))$ as $C$ is a faithfully flat $A$-algebra. As $\dim(C/IC)<\dim(A/I)$, we conclude by the induction hypothesis that $H^n_{IC}(C)=0$, which is a contradiction as above.
    \end{enumerate}
    Hence, $\Supp H^n_I(A)\subseteq \{m\}.$
    
   Let $P_1,\dots,P_t$ be the minimal primes of $I$. Suppose \begin{align*}
        I_1=\bigcap_{p\in P_i}P_i \; \text{ and } \;
        I_2=\bigcap_{p\notin P_i}P_i.
    \end{align*} By Lemma \ref{lemma:colon} and our hypothesis, $\height(I)<\height(I_1).$

    
    Let $D:=A/I\ctens_V A\cong V[[X_2,\dots,X_e,Y_2,\dots,Y_e]]/I(X_2,\dots X_e)$ and $\Delta=(X_2-Y_2,\dots,X_e-Y_e)\subseteq D$. Note that $0\to I\ctens_V A\to A\ctens_V A\to 0$ is a flat resolution of $D$ over $A$. We have a spectral sequence for composition of derived functors $$E_2^{i,j}=h^i(R\Gamma_\Delta(D\dtens_A h^j(R\Gamma_I(A)))).$$ We have that $E_2^{d-1,\:n}\implies E_\infty^{d-1,\:n}$, which is a subquotient in a filtration of
    \begin{align*}
        h^{d-1+n}(R\Gamma_\Delta(D\dtens_A R\Gamma_I(A)))=h^{d-1+n}(R\Gamma_{\Delta+ID}(D)).
    \end{align*} Observe that $ID\subseteq\Delta$ as $D/\Delta\cong A/I$, and $d-1+n>d-1+c\geq e-1,$ which is the number of generators of $\Delta$. Hence, 
    \begin{align*}
        h^{d-1+n}(R\Gamma_{\Delta+ID}(D))&=H^{d-1+n}_{\Delta+ID}(D)\\&=H^{d-1+n}_\Delta(D)\\&=0.
    \end{align*} 
    Therefore, $E_\infty^{d-1,\:n}=0.$ We now show that the incoming differentials to and outgoing differentials from $E_2^{d-1,\:n}$ are $0$, thus establishing that $E_2^{d-1,\:n}=0.$ The incoming differentials are from $E_2^{d-1-r,\:n+r-1}=H_\Delta^{d-1-r}(D\dtens_A H_I^{n+r-1}(A))$ for $r>1$. As $n+r-1>n$, by the induction hypothesis, $H_I^{n+r-1}(A)=0$ and thus $E_2^{d-1-r,\:n+r-1}=0$. 

    The outgoing differentials are to $E_2^{d-1+r,\: n-r+1}$ for $r>1$. If $r>c$, $d-1+r>e-1$ and $E_2^{d-1+r,\:n-r+1}=H_\Delta^{d-1+r}(D\dtens_A H_I^{n-r+1}(A))=0$ by considering a spectral sequence computing $D\dtens_A H_I^{n-r+1}(A)$ and noting that $\Delta$ is generated by $e-1$ elements. Suppose $r\leq c.$ Let $s=r-1$ and $L$ be a finitely generated submodule of $H_I^{n-r+1}(A)$. Note that $1\leq s\leq c-1.$ By the hypothesis of the theorem, $\dim(L)\leq s.$ Then, by \cite[Lemma 2.1(a)]{Skalit}, $$\dim(A/(I_1+I_2)\ctens_V L)\leq\dim(A/I_1\ctens_V L)\leq \dim(A/I_1)+\dim(L)<d+s.$$ Here, we use our assumption that $\height(I)<\height(I_1).$ On the other hand, by \cite[Lemma 2.1(b)]{Skalit}, as $A/I_2$ is $p$-torsion-free, $$\dim(A/I_2\ctens_VL)=\dim(A/I_2)+\dim(L)-1=d+s-1.$$ 
    Consider the exact sequence $$\ctor_1^V(A/(I_1+I_2),L)\to A/I\ctens_V L\to (A/I_1\oplus A/I_2)\ctens_VL\to (A/(I_1+I_2)\ctens_VL)\to0.$$
    As $\dim(\ctor_1^V(A/(I_1+I_2),L))\leq\dim(A/(I_1+I_2)\ctens_VL)$, we have $$\dim(A/I\ctens_VL)\leq\max\{\dim(A/I_1\ctens_VL),\dim(A/I_2\ctens_VL)\}<d+s.$$ Further, observe that $\dim(\ctor_1^V(A/I,L))\leq \dim(A/I\ctens_VL)<d+s.$ Note that as a complex, $D\dtens_A L$ is represented by $$0\to I\ctens_V L\to A\ctens_V L\to 0,$$ with cohomologies $A/I\ctens_VL$ and $\ctor_1^V(A/I,L)$. Therefore, $H_\Delta^{d+s}(D\dtens_A L)=0$. As tensor products and local cohomology commute with direct limits, $E_2^{d-1+r,\: n+r-1}=0.$ We have thus proved $$E_2^{d-1,\:n}=H_\Delta^{d-1}(D\dtens_AH_I^n(A))=0.$$

   It follows from \cite[Theorem 5.1]{Zhou} that the injective dimension of $H_I^n(A)$ as an $A$-module is at most 1 and by \cite[Theorem 1]{Lyu4} or \cite[Theorem 5.1]{NB}, the Bass numbers of $H_I^n(A)$ are finite. We observe that 
    \begin{align*}
        H_\Delta^{d-1}(D\dtens_A E_A(k))&\cong H_\Delta^{d-1}(D\dtens_A H_\mathfrak{m}^e (A))\\
        &\cong h^{d-1}(R\Gamma_\Delta(D\dtens_A R\Gamma_\mathfrak{m}(A)[e]))\\
        &\cong h^{d-1}(R\Gamma_\Delta(R\Gamma_{\mathfrak{m}D}(D)[e]))\\
        &\cong h^{d-1+e}(R\Gamma_{\Delta+\mathfrak{m}D}(D))\\
        &\cong H_{\Delta +\mathfrak{m}D}^{d+e-1}(D),
    \end{align*}
    which is nonzero because $\Delta D+\mathfrak{m}D$ is the maximal ideal of $D$ and $\dim(D)=\dim(A/I)+\dim(A)-1=d+e-1$. As $\Supp(H_I^n(A))\subseteq\{m\}$, if $H_I^n(A)$ is injective, $H_I^n(A)\cong E_A(k)^{\oplus a}$, a contradiction to $H_\Delta^{d-1}(D\dtens_AH_I^n(A))=0$ by the above calculation. Hence, $H_I^n(A)$ must have injective dimension $1$.

    Suppose $$0\to H_I^n(A)\to E_A(k)^{\oplus a}\to E_A(k)^{\oplus b}\to 0$$ is a minimal injective resolution of $H_I^n(A).$ On applying the derived functor $R\Gamma_\Delta(D\dtens_A \underline{\hspace{3mm}})$, we get the following snippet of a long exact sequence: $$h^{d-1}(R\Gamma_\Delta(D\dtens_A H_I^n(A)))\to h^{d-1}(R\Gamma_\Delta(D\dtens_A E_A(k)^{\oplus a}))\to h^{d-1}(R\Gamma_\Delta(D\dtens_A E_A(k)^{\oplus b})).$$ Since $H_\Delta^{d-1}(D\dtens_A H_I^n(A))=0$ and $H_\Delta^{d-1}(D\dtens_A E_A(k))\cong H_{\Delta+\mathfrak{m}D}^{d+e-1}(D),$ we obtain an exact sequence $$0\to H_{\Delta+\mathfrak{m}D}^{e+d-1}(D)^{\oplus a}\to H_{\Delta+\mathfrak{m}D}^{e+d-1}(D)^{\oplus b}$$ with the entries of the map in the ideal $\mathfrak{m}D$ since we start with a minimal resolution. Taking the Matlis dual, we get a surjection $$\omega_D^{\oplus b}\onto \omega_D^{\oplus a},$$ with all entries in the ideal $\mathfrak{m}D$, a contradiction to Nakayama's lemma.
\end{proof}

\section{Bounds on cohomological dimension}\label{section:cdbound}

In this section, we prove our claimed upper bound on cohomological dimension, using Theorem \ref{thm:induction}, \cite[Theorem 4.1]{Lyu1} and a variety of well-known cohomological tools. We begin by laying the groundwork for the proof in the following key lemmas.

\begin{lemma}\label{lemma:radical(I+p)}
    Let $I$ be an ideal of a Noetherian ring $A$. Suppose $x\in A$ is a zerodivisor on $A/I$ and $I=I_1\cap I_2$ such that $x\in I_1$ and $x\notin I_2.$ Then \begin{enumerate}
        \item $\sqrt{I+(x)}=\sqrt{I_1}\cap\sqrt{I_2+(x)}.$
        \item Suppose $x$ is a nonzerodivisor on $A/\sqrt{I_2}$. Then $$\bight(I+(x))\leq\max\{\bight(I_1),\bight(I_2)+1\}.$$
    \end{enumerate}
\end{lemma}
\begin{proof} We verify each assertion in turn.
    \begin{enumerate}
        \item As $I+(x)=I_1\cap I_2 +(x)\subseteq I_1\cap(I_2+(x))$, we have $\sqrt{I+(x)}\subseteq\sqrt{I_1}\cap\sqrt{I_2+(x)}.$ On the other hand, 
        \begin{align*}
            \sqrt{I_1}\cap\sqrt{I_2+(x)}&=\sqrt{(I_1(I_2+(x))}\\
            &\subseteq\sqrt{I_1I_2+(x)}\\
            &\subseteq \sqrt{I+(x)}.
        \end{align*}
        
        \item For any minimal prime $Q$ of $I_2$, as $x\notin Q$, $\height(Q+(x))=\bight(Q+(x))=\height(Q)+1.$ This implies $\bight(I_2+(x))\leq\bight(I_2)+1$. The conclusion follows from (1) and Remark \ref{rmk:bight}.\qedhere
    \end{enumerate}
\end{proof}

\begin{lemma}\label{lemma:p_zd_on_H_I^n}
    Let $I$, $A$, $x$, $I_1$ and $I_2$ be as in Lemma \ref{lemma:radical(I+p)}, $M$ be an $A$-module and $n>\max(\cd(M,I_1),\cd(M,I_2))$. Then $H_I^n(M)$ is $x$-power torsion, that is, every element of $H^n_I(M)$ is annihilated by some power of $x$.
\end{lemma}
\begin{proof}
    By the Mayer-Vietoris sequence on local cohomology, $$H_{I_1}^n(M)\oplus H_{I_2}^n(M)\to H_I^n(M)\to H_{I_1+I_2}^{n+1}(M)\to H_{I_1}^{n+1}(M)\oplus H_{I_2}^{n+1}(M)$$ is exact. As $n>\max(\cd(M,I_1),\cd(M,I_2))$, $H_I^n(M)\cong H_{I_1+I_2}^{n_1+n_2}(M)$. This concludes the proof as $x\in I_1+I_2$.  
\end{proof}

We will crucially use two more results from the literature and we state them for the sake of easy reference.

\begin{theorem}[{\cite[Chapter V, Theorem 3]{Serre}}]\label{thm:Serre}
    Let $A$ be a regular ring, and $P$ and $Q$ be prime ideals of $A$. Then $\height(P+Q)\leq\height(P)+\height(Q).$
\end{theorem}

\begin{theorem}[{\cite[Theorem 4.1]{Lyu1}}]\label{thm:Lyu}
    Let $A$ be a $d$-dimensional regular local ring in equal characteristic with separably closed residue field, and $I$ be an ideal of $A$ with minimal primes $P_1,\dots,P_t.$ Assume that $d>1$ and $\bight(I)=c>0.$ Let $u=\floor*{\frac{d-2}{c}}$ and $v=d-u$. Then $H_I^{v}(A)$ is isomorphic to the cokernel of the map  
    $$ \Phi_I:\bigoplus_{j_0<\cdots<j_{u+1}}H^{d}_{P_{j_0}+\dots+P_{j_{u+1}}}(A)\to \bigoplus_{j_0<\dots<j_{u}}H^{d}_{P_{j_0}+\dots+P_{j_{u}}}(A)$$ 
    that sends every $x\in \bigoplus_{j_0<\cdots<j_{u+1}}H^{d}_{P_{j_0}+\dots+P_{j_u}}(A)$ to $\oplus_{s=0}^{s=u}(-1)^sh_s(x)$ where $$h_s:H^{d}_{P_{j_0}+\dots+P_{j_{u+1}}}(A)\to H^{d}_{P_{j_0}+\dots+\widehat{P_{j_s}}+\dots+P_{j_{u+1}}}(A)$$ ($\widehat{P_{j_s}}$ means that $P_{j_s}$ has been omitted) is the natural map induced by the containment $$P_{j_0}+\dots+\widehat{P_{j_s}}+\dots+P_{j_{u+1}}\subseteq P_{j_0}+\dots+P_{j_{u+1}}.$$
\end{theorem}

We are now prepared to prove our main theorem.

\begin{theorem}\label{thm:cdbound}
    Let $A$ be a $d$-dimensional unramified regular local ring of mixed characteristic $(0,p)$. Let $I$ be a nonzero proper ideal of $A$ with big height $c$. Then $$\cd(A,I)\leq d-\floor*{\frac{d-1}{c}}.$$
\end{theorem}
\begin{proof} We may assume that $A$ is complete and $I$ is radical as local cohomology and big height remain unchanged. If $c=1$, $I$ is principal and $\cd(A,I)\leq 1$. Thus we assume $c>1.$ We also assume $c<d$ as the claimed bound is immediate if $c=d.$

    \noindent\textbf{Case 1:} $p$ is a nonzerodivisor on $A/I$.
    
    We induce on $d$. There is nothing to prove in the case $d=1$. Assume $d>1.$ We need to prove that $H_I^q(A)=0$ for all $q\geq d+1-\floor*{(d-1)/c}$. Let $n=d+1-\floor*{(d-1)/c}.$ By Theorem \ref{thm:induction}, it suffices to prove that for all integers $s$ such that $1\leq s\leq c-1$ and for all $q\geq n-s,$ we have $H_{IA_P}^q(A_P)=0$ for all $P\in \Spec(A)$ such that $I\subseteq P$ and $\dim(A/P)\geq s+1.$ 
    
    Choose integers $s$ and $q$, and a prime $P$ satisfying the above conditions. Note that $p$ is a nonzerodivisor on $A_P/IA_P$ as well and $\dim(A_P)<\dim(A).$ 
    If $p\in P,$ then by the induction hypothesis,
    \begin{align*}
        \dim(A_P)-\floor*{\frac{\dim(A_P)-1}{\bight(IA_P)}}&\leq \dim(A_P)-\floor*{\frac{\dim(A_P)-1}{c}}\\
        &\leq d-s-1-\floor*{\frac{d-s-2}{c}}
    \end{align*} as $\bight(IA_P)\leq \bight(I).$ If $p\notin P,$ $A_P$ contains a field of characteristic $p$ and we have the same conclusion by \cite[Korollar 2]{Faltings}.  We also note that 
    \begin{align*}
        (n-s)-(d-s-1-\floor*{\frac{d-s-2}{c}})=2-\floor*{\frac{d-1}{c}}+\floor*{\frac{d-s-2}{c}}>0,
    \end{align*} where the last inequality follows from $1\leq s\leq c-1.$
    Hence, $H^q_{IA_P}(A_P)=0$ as required. 

    \noindent\textbf{Case 2:} $p\in I.$ 
    
    As $A$ is unramified, $A/(p)$ is a regular local ring of equicharacteristic $p$. As $\dim(A/(p))=d-1$ and $\bight(I/(p))=c-1$, by \cite[Korollar 2]{Faltings}, $$\cd(A/(p),I/(p))\leq d-1-\floor*{\frac{d-2}{c-1}}\leq d-1-\floor*{\frac{d-1}{c}}.$$ Hence, by the long exact sequence on local cohomology, $$\dots \to H_I^i(A/(p))\to H_I^{i+1}(A)\xlongrightarrow{p}H_I^{i+1}(A)\to H_I^{i+1}(A/(p))\to\dots, $$ we have $$\cd(A,I)\leq \cd(A/(p),I/(p))+1\leq d-\floor*{\frac{d-1}{c}}.$$

    \noindent\textbf{Case 3:} $p$ is a zerodivisor on $A/I$ and $p\not\in I$.
    
    For the treatment of this case, we assume that $A$ has separably closed residue field. We may do so as completion and strict Henselization are faithfully flat, and thus preserve cohomological dimension and big height (Lemma \ref{lemma:fflat}).
    Let $P_1,\dots,P_t$ be the minimal primes of $I$. Suppose \begin{align*}
        I_1=\bigcap_{p\in P_i}P_i \; \text{ and } \;
        I_2=\bigcap_{p\notin P_i}P_i.
    \end{align*}
    Then $I=I_1\cap I_2$, $p\in I_1$, and $p$ is a nonzerodivisor on $A/I_2$. By Lemma \ref{lemma:radical(I+p)}, $$\bight(I+(p))\leq\max\{\bight(I_1),\bight(I_2)+1\}\leq c+1.$$

    
    

    Let $n>d-\floor*{\frac{d-1}{c}}$. Note that as every minimal prime of $I_1$ or $I_2$ is a minimal prime of $I$, we have $\bight(I_1)\leq c$ and $\bight(I_2)\leq c.$ As $p\in I_1$ and $p$ is a nonzerodivisor on $A/I_2$, the cohomological dimensions of $I_1$ and $I_2$ are bounded above by $d-\floor*{\frac{d-1}{c}}$ by Case 2 and Case 1 respectively. By Lemma \ref{lemma:p_zd_on_H_I^n}, we conclude that $H^n_I(A)$ is $p$-power torsion.
    
    Let $$J=\sqrt{I(A/(p))}=\sqrt{I+(p)}/(p),\: J_1=I_1/(p) \text{ and } J_2=\sqrt{I_2(A/(p))}=\sqrt{I_2+(p)}/(p)$$ be ideals in $A/(p)$. We have $\bight(J)=\bight(I+(p))-1\leq c$. Using the exactness of the sequence $$H_J^{n-1}(A/(p))\to H_I^{n}(A)\xlongrightarrow{p}H_I^{n}(A),$$ it suffices to prove $H_J^{n-1}(A/(p))=0.$
    As $A/(p)$ is a regular local ring of dimension $d-1$ containing a field, by \cite[Korollar 2]{Faltings},
    \begin{align*}
        \cd(A/(p),J)&\leq d-1-\floor*{\frac{d-2}{\bight(J)}}\\
                    &\leq d-1-\floor*{\frac{d-2}{c}}\\
                    &\leq d-\floor*{\frac{d-1}{c}}\\
                    &\leq n-1.
    \end{align*}
    \noindent If any of the inequalities above are strict, we are done. For instance, if $c$ does not divide $d-1$, the second inequality above is strict. In the only remaining case, $$n=d-\floor*{\frac{d-1}{c}}+1, \,\;c|(d-1),\,\; \bight(J)=c, \text{ and } \cd(A/(p),J)=n-1.$$ We analyze the strictness of the inequality $d-1-\floor*{\frac{d-2}{c}}\geq \cd(A/(p),J)$ given by Faltings' bound on cohomological dimension. 

    Let $d-1=Nc$. Then $\floor*{\frac{d-1}{c}}=N$ and $\floor*{\frac{d-2}{c}}=N-1$.
    By Theorem \ref{thm:Lyu}, as $A$ has a separably closed residue field, $H^{d-N}_J(A/(p))$ is isomorphic to the cokernel of the map 
    $$ \Phi_J:\bigoplus_{j_0<\cdots<j_N}H^{d-1}_{P_{j_0}+\dots+P_{j_N}}(A/(p))\to \bigoplus_{j_0<\dots<j_{N-1}}H^{d-1}_{P_{j_0}+\dots+P_{j_{N-1}}}(A/(p))$$ 
    that sends every $x\in \bigoplus_{j_0<\cdots<j_N}H^{d-1}_{P_{j_0}+\dots+P_{j_N}}(A/(p))$ to $\oplus_{s=0}^{s=N-1}(-1)^sh_s(x)$ where $$h_s:H^{d-1}_{P_{j_0}+\dots+P_{j_N}}(A/(p))\to H^{d-1}_{P_{j_0}+\dots+\widehat{P_{j_s}}+\dots+P_{j_N}}(A/(p))$$ ($\widehat{P_{j_s}}$ means that $P_{j_s}$ has been omitted) is the natural map induced by the containment $$P_{j_0}+\dots+\widehat{P_{j_s}}+\dots+P_{j_N}\subseteq P_{j_0}+\dots+P_{j_N}.$$
    Note that if $\bight(J)=c,$ then $\bight(I)=\bight(I_2)=c+1$ by Lemma \ref{lemma:radical(I+p)}. Hence, $\bight(J_1)=\bight(I_1)-1<c$ and $\bight(J_2)=\bight(I_2+(p))-1=c$. Observe that if any of $P_{j_0},\dots,P_{j_{N-1}}$ is a minimal prime of $J_1$, 
    $$\height(P_{j_0}+\dots+P_{j_{N-1}})\leq (c-1)+(N-1)c=d-2$$ by Theorem \ref{thm:Serre}. Hence, $P_{j_0}+\dots+P_{j_{N-1}}$ cannot be primary to the maximal ideal of $A/(p)$ and 
    $$H^{d-1}_{P_{j_0}+\dots+P_{j_{N-1}}}(A/(p))=0$$ by the Hartshorne--Lichtenbaum Vanishing Theorem. Without loss of generality, assume that $P_1,\dots,P_s$ are the minimal primes of $J_2$. The above calculations ascertain that the codomain of $\Phi_J$ is $$\bigoplus_{j_0<\dots<j_{N-1}\leq s}H^{d-1}_{P_{j_0}+\dots+P_{j_{N-1}}}(A/(p)).$$
    As $\bight(J_2)=c$, Theorem \ref{thm:Lyu} applies similarly and $H^{d-n}_{J_2}(A/(p))$ is isomorphic to the cokernel of the natural map $$ \Phi_{J_2}:\bigoplus_{j_0<\cdots<j_N\leq s}H^{d-1}_{P_{j_0}+\dots+P_{j_N}}(A/(p))\to \bigoplus_{j_0<\dots<j_{N-1}\leq s}H^{d-1}_{P_{j_0}+\dots+P_{j_{N-1}}}(A/(p)).$$
    Hence, the codomain of $\Phi_J$ is equal to the codomain of $\Phi_{J_2}$. Further, it is clear that $\image(\Phi_{J_2})\subseteq\image(\Phi_{J}).$ We conclude that the natural map 
    $$H_{J_2}^{d-N}(A/(p))\to H_{J}^{d-N}(A/(p))$$ induced by the inclusion $J\subseteq J_2$ is surjective. 
    
    Consider the commutative diagram
    \[\begin{tikzcd}[cramped,sep=tiny]
	{H_{I_2}^{d-N}(A)} & {\xlongrightarrow{f_1}} & {H_{J_2}^{d-N}(A/(p))} & {\xlongrightarrow{f_2}} & {H_{I_2}^{d-N+1}(A)} \\
	{\Big\downarrow{\scriptstyle \alpha}} && {\Big\downarrow{\scriptstyle \beta}} && {\Big\downarrow}\\
	{H_{I}^{d-N}(A)} & {\xlongrightarrow{g_1}} & {H_{J}^{d-N}(A/(p))} & {\xlongrightarrow{g_2}} & {H_{I}^{d-N+1}(A)} & {\xlongrightarrow{p}} & {H_{I}^{d-N+1}(A)}
    \end{tikzcd}\]
    We showed that $\beta$ is surjective. Since $p$ is a nonzerodivisor on $A/I_2$, we have $\cd(A,I_2)\leq d-N$ by Case 1. Therefore, $H_{I_2}^{d-N+1}(A)=0$ and $f_1$ is surjective. Consequently, $\beta f_1=g_1\alpha$ is surjective, forcing $g_1$ to be surjective and $g_2=0.$ Hence, multiplication by $p$ on $H_I^{d-N+1}(A)$ is injective and we infer that $H_I^{d-N+1}(A)=0$ as $H_I^{d-N+1}(A)=H_I^n(A)$ is $p$-power torsion by Lemma \ref{lemma:p_zd_on_H_I^n}. We have thus shown that \[\cd(A,I)\leq d-N=d-\floor*{\frac{d-1}{c}}.\qedhere\]
\end{proof}

\begin{remark}
    Note that in the case $p\in I$ in the above proof, we in fact show a stronger bound for $\cd(A,I)$, namely that $$\cd(A,I)\leq\cd(A/p,I/p)+1\leq d-\floor*{\frac{d-2}{c-1}}.$$ This bound may differ from the bound of Theorem \ref{thm:cdbound}: for example, if $d=5$ and $c=3$. This may lead one to question the general sharpness of the bound on cohomological dimension in Theorem \ref{thm:cdbound}. We address the sharpness in the following example.
\end{remark}

\begin{example}\label{example:sharpness}
    Given an unramified regular local ring $(A,\mathfrak{m})$ of dimension $d>0$ and a positive integer $c\leq d$, there exists an ideal $I$ of $A$ such that $\cd(A,I)=d-\floor*{\frac{d-1}{c}}.$ We construct this ideal as in \cite{Lyu2}. 

    Let $N=\floor*{\frac{d-1}{c}}.$ Let $I_0,\dots,I_N$ be ideals of pure height $c$ in $A$ such that $I_0+\dots+I_N$ is $\mathfrak{m}$-primary and let $I=I_0\cap\dots\cap I_N$. Then $\cd(A,I)=d-N.$ The proof of this statement is identical to \cite{Lyu2}.

    To resolve any conflict with the previous remark, we show that in this example, if $p\in I$, then $\floor*{\frac{d-2}{c-1}}= \floor*{\frac{d-1}{c}}.$ Indeed, if $p\in I,$ we have $p\in I_j$ for all $j\in\{0,\dots,N\}$ and by Theorem \ref{thm:Serre}, $$d=\height(I_0+\dots+I_N)\leq 1+(N+1)(c-1)=Nc+c-N.$$ Let $d-1=Nc+k$, where $0\leq k<c$. Then the above equation yields $N+k+1\leq c$, or $N+k-1\leq c-2$. Observe that $$d-2=N(c-1)+N+k-1,$$ and hence, $\floor*{\frac{d-2}{c-1}}=N=\floor*{\frac{d-1}{c}}$. 
\end{example}

\begin{remark}
    In \cite{HNBPW}, the authors attributed Theorem \ref{thm:cdbound} to Faltings and utilised it to derive consequences of their results. In subsequent correspondence, the authors clarified that this attribution was based on a misinterpretation of Faltings’ work, and that a proof of Theorem \ref{thm:cdbound} does not appear to be available in the literature.

    The argument presented above recovers their result \cite[Theorem 3.11]{HNBPW}. In fact, we obtain a strengthened version below by exploiting the Second Vanishing Theorem of Zhang \cite{Zhang} in unramified mixed characteristic.
\end{remark}

\begin{theorem}[{\cite[Theorem 1.4]{Zhang}}] \label{thm:secondvanishing}
    Let $A$ be an unramified $d$-dimensional regular local ring of mixed characteristic with separably closed residue field. Then, for each ideal $I$ of $A$, the following conditions are equivalent:
    \begin{enumerate}
        \item $H_I^j(A)=0$ for all $j>d-2.$
        \item $\dim(A/I)\geq 2$ and the punctured spectrum of $A/I$ is connected.
    \end{enumerate}
\end{theorem}

\begin{corollary}[cf. {\cite[Lemma]{Lyu2}}, {\cite[Lemma 3.10]{HNBPW}}]\label{cor:Lyu}
    Let $A$ be a $d$-dimensional regular local ring of equal characteristic or unramified mixed characteristic. Let $I_1,\dots,I_N$ be ideals of $A$ of big height at most $c\geq 1$. Then $$\cd(A,I_1+\dots+I_N)\leq d-\floor*{\frac{d-1}{c}}+N-1.$$
\end{corollary}
\begin{proof}
    We proceed by induction on $N$. The case $N=1$ follows from \cite[Satz 1]{Faltings} in equal characteristic and Theorem \ref{thm:cdbound} in unramified mixed characteristic. For $N>1,$ let $I'=I_1+\dots+I_{N-1}$, $I=I'+I_N$ and $J=I_1\cap I_t+\dots+I_{N-1}\cap I_N$. Note that each $I_j\cap I_t$ has big height at most $c$ by Remark \ref{rmk:bight}. As $\sqrt{J}=\sqrt{I'\cap I_N}$, the Mayer-Vietoris sequence for local cohomology yields $$\cdots\to H_J^{i-1}(A)\to H_I^{i}(A)\to H_{I'}^{i}(A)\oplus H_{I_N}^{i}(A)\to \cdots$$ 
    For $i>d-\floor*{\frac{d-1}{c}}+N-1$, $H_J^{i-1}(A)=H_{I'}^{i}(A)=H_{I_N}^{i}(A)=0$ by the induction hypothesis. Hence, $H_I^i(A)=0$ for $i>d-\floor*{\frac{d-1}{c}}+N-1$.
\end{proof}

\begin{theorem}[cf. {\cite[Theorem 5.2]{HunLyu}}]\label{thm:example}
    Let $(A,\mathfrak{m})$ be a $d$-dimensional regular local ring of equal characteristic or unramified mixed characteristic. Let $c<d$ be a positive integer and set $N=\floor*{\frac{d-1}{c}}$. Let $I_1,\dots,I_N$ be ideals of $A$ of big height at most $c$ such that $\Spec(A/(I_1+\dots+I_N))-\{\mathfrak{m}\}$ is disconnected. Then $\cd(A,I_1\cap\dots\cap I_N)=d-N.$ 
\end{theorem}
\begin{proof} 
    If $(B,\mathfrak{n})$ is a faithfully flat $A$-algebra with zero-dimensional fiber, then $\bight(I_jB)=\bight(I)$  (Lemma \ref{lemma:fflat}) and $\Spec(B/(I_1B+\cdots+I_NB)-\{\mathfrak{n}\}$ is disconnected. Hence, we may assume that $A$ has separably closed residue field.
    
    Let $\mathfrak{a}_i=(I_1\cap\dots\cap I_i) + I_{i+1}+\dots+I_N.$
    By Corollary \ref{cor:Lyu},
    $$\cd(A,\mathfrak{a}_i)\leq d-N+(N-i+1)-1=d-i.$$
    We show by induction that $\cd(A,\mathfrak{a}_i)=d-i.$ For $i=1$, as the punctured spectrum of
    $A/\mathfrak{a}_1$ is disconnected, this follows
    from \cite[Theorem 2.9]{HunLyu} in equal characteristic and Theorem \ref{thm:secondvanishing} in unramified mixed characteristic. For $i>1$, let
    $$J_1=I_i+\dots+I_N \text{ and } J_2=(I_1\cap\dots\cap I_{i-1})+I_{i+1}+\dots+I_N.$$
    Then $\sqrt{J_1+J_2}=\sqrt{\mathfrak{a}_{i-1}}$ and
    $\sqrt{J_1\cap J_2}=\sqrt{\mathfrak{a}_i}.$ Observe that $\cd(A,J_s)\leq d-i$ for $s=1,2$ by Lemma \ref{cor:Lyu}. 
    The associated Mayer-Vietoris sequence for local cohomology
    $$\cdots\to H_{\mathfrak{a}_i}^{d-i}(A)\to H_{\mathfrak{a}_{i-1}}^{d-i+1}(A)\to H_{J_1}^{d-i+1}(A)\oplus H_{J_2}^{d-i+1}(A)\to\cdots$$ implies that $H_{\mathfrak{a}_i}^{d-i}(A)\neq 0$ as $\cd(A,\mathfrak{a}_{i-1})=d-i+1$ by the induction hypothesis. Hence, $\cd(A,\mathfrak{a}_i)=d-i.$ The conclusion follows as $\mathfrak{a}_N=I_1\cap\dots\cap I_N.$
\end{proof}

We elaborate on an example of Huneke and Lyubeznik to show that Theorem \ref{thm:example} addresses examples not covered by Example \ref{example:sharpness}.

\begin{example}[{\cite[Example 5.4]{HunLyu}}]
    Let $(A,\mathfrak{m})$ be a $d$-dimensional regular local ring of equal characteristic or unramified mixed characteristic. Suppose $P_1,\dots,P_5$ are prime ideals of $A$ of height $c$ such that $\floor*{\frac{d-1}{c}}=2.$ Further, suppose that $P_i+P_j+P_k$ is $\mathfrak{m}$-primary iff $$\{i,j,k\}\in\Lambda=\{(1,3,4),(1,3,5),(1,2,5),(2,3,4),(2,4,5)\}.$$ Let $I=P_1\cap\cdots\cap P_5.$
    For instance, we may consider $A=V[[x_2,\dots,x_9]]$, where $V$ is a discrete valuation ring with uniformizer $p$ and separably closed residue field, and $$P_1=(p,x_2,x_3,x_4), \;P_2=(x_3,x_4,x_5,x_6),\; P_3=(x_5,x_6,x_7,x_8),$$ $$P_4=(p,x_2,x_8,x_9) \text{ and } P_5=(x_6,x_7,x_8,x_9).$$
    Note that $\Lambda$ has the property that each $i$, $1\leq i\leq5$ appears no more than thrice and each pair $(i,j)$ appears no more than twice. This implies that we cannot split $\{1,\dots,5\}$ into three distinct subsets $S_0$, $S_1$ and $S_2$ such that $P_i+P_j+P_k$ is $\mathfrak{m}$-primary for all $i\in S_0$, $j\in S_1$, $k\in S_2$. Hence, we are unable to apply the result of Example \ref{example:sharpness}.

    However, we may write $I=I_1\cap I_2$, where $I_1=P_1\cap P_2\cap P_3$ and $I_2=P_4\cap P_5$. To apply Theorem \ref{thm:example}, we need to check that $\Spec(A/(I_1+I_2))-\mathfrak{m}$ is disconnected. Let $J=(P_1+P_4)\cap (P_2+P_4)\cap (P_1+P_5)$ and $K=(P_3+P_4)\cap (P_2+P_5)\cap (P_3+P_5)$. Then $J+K$ is $\mathfrak{m}$-primary and $\sqrt{J\cap K}=\sqrt{I_1+I_2}$. Therefore, the hypotheses of Theorem \ref{thm:example} are satisfied and $\cd(A,I)=d-2.$  
\end{example}

\section{Questions and Further Directions}\label{section:questions}

\begin{definition}
    Let $A$ be a Noetherian local ring and $B=\widehat{((\widehat A)^{sh})}$, the completion of the strict Henselization of the completion of $A$. An ideal $I$ of $A$ is said to be \textit{formally geometrically irreducible} if $IB$ has a unique minimal prime in $B$.
\end{definition}

In 1990, Huneke and Lyubeznik \cite{HunLyu} improved on Faltings' bound in equal characteristic under the assumption that the ideal in question is formally geometrically irreducible.

\begin{theorem}[{\cite[Theorem 3.8]{HunLyu}}]\label{thm:HunLyu_irred}
    Let $A$ be a $d$-dimensional regular local ring of equal characteristic and $I$ be a formally geometrically irreducible ideal of $A$ of big height $c$. Suppose $0<c<d.$ Then $$\cd(A,I)\leq d-1-\floor*{\frac{d-2}{c}}.$$
\end{theorem}

Huneke and Lyubeznik proved this bound by improving on \cite[Satz 1]{Faltings} if $I$ is formally geometrically irreducible.\footnote{In fact, Huneke and Lyubeznik proved an even stronger inductive criterion in equal characteristic; see \cite[Theorem 2.5]{HunLyu}.} 

\begin{theorem}[{\cite[Lemma 3.7]{HunLyu}}]
    Let $A$ be a regular local ring of equal characteristic and $I$ be a formally geometrically irreducible ideal of $A$. 

    Set $c=\bight(I)$. Let $n>c$ be an integer. Assume that for all integers $s$, with $1\leq s\leq c-1,$ and for all $q\geq n-s$, $H^q_{IA_P}(A_P)=0$ for all $P\in \Spec(A)$ such that $I\subseteq P$ and $\dim(A/P)\geq s+2.$ Then $H^q_I(A)=0$ for all $q\geq n.$
\end{theorem}

\begin{question}\label{question:irredbound}
    Let $A$ be a $d$-dimensional regular local ring of unramified mixed characteristic and $I$ be a formally geometrically irreducible ideal of $A$ of big height $c$. Suppose $0<c<d$. Is $$\cd(A,I)\leq d-1-\floor*{\frac{d-2}{c}}?$$
\end{question}

\begin{proposition}
    Given $A$, $I$, $d$ and $c$ as above, if $p$ is in $I$ or a zerodivisor on $A/I$, then $$\cd(A,I)\leq d-1-\floor*{\frac{d-2}{c}}.$$ 
\end{proposition}
\begin{proof}
    If $c=1$, $I$ is principal and $\cd(A,I)\leq 1\leq d-1-\floor*{\frac{d-2}{c}}$ is immediate. Assume $c>1.$
    
    We may assume that $A$ is complete and has separably closed residue field. Then, as $I$ is formally geometrically irreducible, it has a unique minimal prime. We may assume that $I$ is radical and hence prime. Now we have $p\in I.$
    
    As $A$ is unramified, $A/(p)$ is a complete regular local ring of equal characteristic $p$. Observe that $\dim(A/(p))=d-1$, $\bight(I/(p))=c-1$ and $I/(p)$ has a unique minimal prime in $A/(p).$ By Theorem \ref{thm:HunLyu_irred}, $$\cd(A/(p),I/(p))\leq d-2-\floor*{\frac{d-3}{c-1}}.$$ Consider the long exact sequence of local cohomology modules $$\cdots\to H_I^i(A/(p))\to H_I^{i+1}(A)\xlongrightarrow{p} H_I^{i+1}(A)\to\cdots.$$ As $p\in I$, $H_I^i(A)$ is $p$-power torsion for all $i$. Therefore, \[\cd(A,I)\leq\cd(A/(p),I/(p))+1\leq d-1-\floor*{\frac{d-3}{c-1}}\leq d-1-\floor*{\frac{d-2}{c}}.\qedhere\]
\end{proof}

Question \ref{question:irredbound} remains unanswered if $p$ is a nonzerodivisor on $A/I.$ Lastly, we pose:

\begin{question}\label{question:ramified}
    Let $A$ be a $d$-dimensional ramified regular local ring of mixed characteristic $(0,p)$. Let $I\neq0$ be an ideal of $A$ of big height $c$. Is $\cd(A,I)\leq d-\floor*{\frac{d-1}{c}}?$
\end{question}

In the following remark, we discuss a possible approach to Question \ref{question:ramified}, and attempt to use the sharpness of the bound in unramified mixed characteristic to answer it in the negative.

\begin{remark}
    Given $A$, $I$, $d$ and $c$ as above, we may assume $A$ is complete in our analysis of Question \ref{question:ramified}. By Cohen's Structure Theorem, $$A=\frac{V[[x_1,\dots,x_d]]}{(g-p)},$$ where $V$ is a discrete valuation ring with uniformizer $p$ and $g\in (p,x_1,\dots,x_d)^2\setminus (p)$. Let $B=V[[x_1,\dots,x_d]]$ and $f=g-p\in B.$ 

    Let $J=I+(f)$ be an ideal of $B$. By Lemma \ref{lemma:radical(I+p)}, $\bight(J)\leq\bight(I)+1\leq c+1.$ By Theorem \ref{thm:cdbound}, we have that $$\cd(B,J)\leq d+1-\floor*{\frac{d}{c+1}}.$$ Consider the long exact sequence of local cohomology modules $$H^{n-1}_J(B)\xlongrightarrow{f}H^{n-1}_J(B)\to H_I^{n-1}(A)\to H^n_J(B)\xlongrightarrow{f}H^n_J(B).$$
    Observe that if $f\in J$ and $H^n_J(B)\neq0$, then $H_I^{n-1}(A)\neq0$. Thus, to obtain an example such that $\cd(A,I)>d-\floor*{\frac{d-1}{c}}$ via this approach, we need to choose $d$ and $c$ such that $$\bigg(d+1-\floor*{\frac{d}{c+1}}\bigg)-1>\bigg(d-\floor*{\frac{d-1}{c}}\bigg),$$ which reduces to $$\floor*{\frac{d-1}{c}}>\floor*{\frac{d}{c+1}}.$$ It remains to find an ideal $J$ in $B$ such that $f\in J$, $\bight(J)=c+1$ and $\cd(A,J)=d+1-\floor*{\frac{d}{c+1}}.$ Indeed, for such a $J$, $\cd(A,J/(f))\geq \bigg(d+1-\floor*{\frac{d}{c+1}}\bigg)-1>\bigg(d-\floor*{\frac{d-1}{c}}\bigg).$

    Let $N=\floor*{\frac{d-1}{c}}$ and $M=\floor*{\frac{d}{c+1}}$. We attempt to construct a $J$ as in Example \ref{example:sharpness}. We need ideals $I_0,\dots,I_M\subseteq B$ of big height $c+1$ such that $f\in I_j$ for all $j$ and $I_0+\dots+I_M$ is $(p,x_1,x_2,\dots,x_d)$-primary. However, by Theorem \ref{thm:Serre} applied to the ideal $(I_0+\dots+I_M)/(f)$ of $A$, 
    \begin{align*}
        \height(I_0+\dots+I_M)&\leq 1+(M+1)c\\
        &\leq 1+Nc\\
        &\leq 1+(d-1)\\
        &<d+1.
    \end{align*}
    Hence, such a choice of $I_0,\dots,I_M$ does \emph{not} exist, and we do not obtain a counterexample via this approach.

    To the best of our knowledge, Question \ref{question:ramified} remains unanswered.
\end{remark}

\section*{Acknowledgments}
I am grateful to Linquan Ma for many insightful discussions and helpful suggestions. I thank Uli Walther for his perceptive comments and advice that improved the manuscript. I am indebted to Vaibhav Pandey for numerous enlightening conversations, and for his constant support and encouragement. I also thank Karl Schwede and Bernd Ulrich for valuable discussions concerning this work.

A substantial portion of this work was completed while I was visiting the Institute for Advanced Study in Princeton. I thank IAS for their warm hospitality and Linquan Ma for the invitation to visit.

This work was supported by NSF Grants DMS-2302430 and DMS-2100288, and by Simons Foundation Grant SFI-MPS-TSM-00012928.

\bibliographystyle{alpha}
\bibliography{main.bib}

\end{document}